\documentclass[11pt]{article}

\title{\Large\bf Jensen's inequality in geodesic spaces with lower bounded curvature}
\author{Quentin Paris\footnote{HSE University, Faculty of Computer Science, Moscow, Russia. This work has been funded by the  Russian Academic Excellence Project '5-100'. Email:\url{qparis@hse.ru}}}
\date{}

\usepackage{amsthm,amssymb,amsmath,eucal}
\usepackage{color}
\usepackage{enumitem}
\usepackage{url}
\usepackage{bbold}
\usepackage{cases}
\usepackage{float}

\usepackage{titlesec}
\titleformat{\section}
{\large\bf}{\thesection}{1em}{}
\titleformat{\subsection}
{\normalfont\bf}{\thesubsection}{1em}{}

\usepackage[square,numbers]{natbib}
\usepackage[pdftex]{hyperref}
\usepackage[hyperpageref]{backref}
\setcitestyle{authoryear,open={(},close={)}}

\usepackage{tocloft}

\numberwithin{equation}{section}
\usepackage{hyperref}
\hypersetup{
  colorlinks = true,
  urlcolor = blue, 
  linkcolor = blue,
  citecolor = blue}

\newtheorem{thm}{Theorem}[section]

\newtheorem{lem}[thm]{Lemma}

\newtheorem{defi}[thm]{Definition}
\newtheorem{rem}[thm]{Remark}

\newcommand{\R}{\mathbb R}

\newcommand{\e}{\varepsilon}

\newcommand{\curv}{\mathrm{curv}}

\begin{document}

\maketitle
\begin{abstract}
Let $(M,d)$ be a separable and complete geodesic space with curvature lower bounded, by $\kappa\in \R$, in the sense of Alexandrov. Let $\mu$ be a Borel probability measure on $M$, such that $\mu\in\mathcal P_2(M)$, and that has at least one barycenter $x^{*}\in M$. We show that for any geodesically $\alpha$-convex function $f:M\to \R$, for $\alpha\in \R$, the inequality 
\[f(x^*)\le \int_M (f -\frac{\alpha}{2}d^2(x^*,.))\,{\rm d}\mu,\]
holds provided $f$ is locally Lipschitz at $x^*$ and either positive or in $L^1(\mu)$. Our proof relies on the properties of tangent cones at barycenters and on the existence of gradients for semiconcave functions in spaces with lower bounded curvature.
\end{abstract}

\section{Introduction}

Jensen's inequality encodes a powerful connection between convexity and probability theory with consequences in a substantial part of mathematics.
In its classical form, established in the pioneering paper by~\cite{Jens06}, the result states that given a convex function $f:\R^d\to\R$ and a Borel probability measure $\mu$ on $\R^d$,  
\[f(x^*)\le \int_{\R^d}f\mathrm{d}\mu,\quad\mbox{where}\quad x^*=\int_{\R^d}x\,\mathrm{d}\mu(x),\]
whenever the integrals make sense. Since then, a number of natural extensions of the result have been proposed in infinite dimensional Hilbert or Banach spaces. 

The generalization of Jensen's inequality to more abstract metric spaces, with no linear structure, requires an adapted definition of mean value and convexity. In the context of geodesic spaces, considered in the present paper, such adapted notions are usually given by barycenters and geodesic convexity respectively. We first briefly recall these concepts and refer the reader to Section \ref{sec:prelim} for precise definitions.

Given a metric space $(M,d)$, we denote $\mathcal P_2(M)$ the set of Borel probability measures $\mu$ on $M$ such that, for all $x\in M$,
\[\mathcal V_{\mu}(x):=\int_M d^2(x,y)\,{\rm d}\mu(y)<+\infty.\]
Given $\mu\in\mathcal P_2(M)$, we call $\mathcal V_{\mu}:M\to \R_+$ its variance functional, denote \[\mathcal V^*_{\mu}:=\inf_{x\in M}\mathcal V_{\mu}(x),\] 
and call barycenter of $\mu$ any $x^*\in M$ such that $\mathcal V_{\mu}(x^*)=\mathcal V^*_{\mu}$.  Barycenters provide a natural generalization\footnote{Note for instance that if $(M,d)=(\R^p,\|.-.\|_2)$ and if $\mu\in\mathcal P_2(M)$, then $x^*:=\int x\,\mathrm{d}\mu(x)$ is the unique minimizer of 
\[x\in M\mapsto\int \|x-y\|^2_2\,\mathrm{d}\mu(y).\]} of the notion of mean value of a probability measure when $M$ has no linear structure.
While alternative notions of mean value in a metric space have been proposed, barycenters are often favored for their simple interpretation and constructive definition as the solution of an optimization problem. In full generality, a barycenter may not exist and, if it does, needs not be unique. The questions of existence and uniqueness of barycenters have been addressed in a number of settings and we refer the reader to~\cite{Stu03,AguCar11,Afs11,ohta2012barycenters,Yok16,KimPas17,GouLou17,AhiGouPar19,GouParRigStr19} and~\cite{HucElt20} for discussions on this topic in different settings.

Given a geodesic space $(M,d)$, and $\alpha\in\R$, a function $f:M\to \R$ is called geodesically $\alpha$-convex if, for every geodesic $\gamma:[0,1]\to M$, the function
\[t\in[0,1]\mapsto f(\gamma(t))-\frac{\alpha}{2}d^2(\gamma(0),\gamma(t)),\]
is convex. This notion reduces to the classical definition of convexity (resp. strong convexity) in the Euclidean setting when $\alpha=0$ (resp. $\alpha>0$).

In the context of Riemannian manifolds,~\cite{EmeMok91} establish Jensen's inequality for exponential barycenters. The problem is addressed in~\cite{Ken90,Stu03,Kuw09} and~\cite{Yok16} in the context of more abstract metric spaces with upper curvature bound, in the sense of Alexandrov, and in~\cite{Kuw14} for convex metric spaces. 

Less seems to be known in the context of metric spaces with a lower curvature bound in the sense of Alexandrov (Definition \ref{def:4pc}). The central role of such metric spaces has been greatly emphasized since the celebrated compactness theorem of M. Gromov~\citep{Grom81}, stating that for all integers $n\ge 0$, all $\kappa\in \R$ and all $D>0$, the set of length spaces with Hausdorff dimension at most $n$, curvature lower bounded by $\kappa$ and diameter at most $D$ is compact in the Gromov-Hausdorff metric.

A notable exception is the case of the $2$-Wasserstein space $M=\mathcal W_2(\Omega)$ over a Polish (i.e., separable and complete) space $\Omega$, known to be a geodesic space of curvature lower bounded by $0$ if, and only if, $\Omega$ satisfies the same property~\cite[Proposition 2.10]{Stu06}. In the case where $\Omega=\R^d$, and $\mu$ is a probability measure with finite support over $M$,~\cite{AguCar11} prove Jensen's inequality for a number of classical geodesically convex functionals. This result is later generalized by~\cite{KimPas17} to the case where $\Omega$ is a compact Riemannian manifold and $\mu$ is a sufficiently well behaved Borel probability measure over $M$.

In this paper, we generalize available results in the litterature and prove the validity of Jensen's inequality in general Polish geodesic spaces with curvature lower bounded, by some $\kappa\in \R$, in the sense of Alexandrov. Precisely, we prove the following result.

\begin{thm}
\label{thm:main}
Let $(M,d)$ be a Polish geodesic space with curvature lower bounded by some $\kappa\in \R$, in the sense of Alexandrov (Definition \ref{def:4pc}). Let $\mu\in \mathcal P_2(M)$ and suppose $\mu$ admits a barycenter $x^*\in M$. Let $\alpha\in \R$ and $f:M\to \R$ be geodesically $\alpha$-convex and locally Lipschitz at $x^*$. Then if $f$ is either positive, or in $L^1(\mu)$, we have
\[f(x^*)\le \int_M f\,{\rm d}\mu-\frac{\alpha}{2}\mathcal V^*_{\mu}.\]
\end{thm}

The result calls for a particular comment. On the one hand, the result of~\cite{KimPas17} is more general than Theorem \ref{thm:main}, in the context of the Wasserstein space, since it does not impose restrictions on the sectional curvature of the base space $\Omega$. On the other hand, our result implies the validity of Jensen's inequality in $\mathcal W_2(\Omega)$ as soon as $\Omega$ is a Polish geodesic space with non-negative curvature, which needs not be a smooth manifold.

From a technical point of view, our proof relies on a two essential results from metric geometry. First, we invoke (see Lemma \ref{lem:tpstarisflat}) specific properties of tangent cones at barycenters of probability measures in spaces with lower bounded curvature as studied in~\cite{ohta2012barycenters,Yok12,GouParRigStr19} and~\cite{Gou20}. Second, we rely on (see Lemma \ref{lem:grad}) the existence of an appropriate notion of gradient for semi-concave functions on spaces with lower bounded curvature~\cite[see, e.g.,][]{Pet07}.

The paper is organized as follows. Section \ref{sec:prelim} introduces background in metric geometry, necessary for our result. Section \ref{sec:theproof} presents the proof of Theorem \ref{thm:main}. Finally, Appendix \ref{sec:appendix} details proofs of some key lemmas presented in the preliminary section.

\section{Preliminaries}
\label{sec:prelim}
In this section we summarize some definitions and results from metric geometry necessary for our main result. Most of the material gathered below can be found in classical texts such as~\citet{bgp12,BuBuIv01,plaut02} or in the book in preparation by~\cite{AleKapPet19}. We provide proofs for some of the key statements in Appendix \ref{sec:appendix} to clarify the presentation.

\subsection{Geodesic spaces with lower bounded curvature} 
Let $(M,d)$ be a metric space. We call path in $M$ a continuous map $\gamma:I\to M$ defined on an interval $I\subset \R$. The length $L(\gamma)\in[0,+\infty]$ of a path $\gamma:I\to M$ is defined by 
\begin{equation}
 \label{def:length}
 L(\gamma):=\sup_{n\ge 1,\,t_0\le\dots\le t_{n}\in I}\sum_{i=0}^{n-1}d(\gamma(t_i),\gamma(t_{i+1})).
 \end{equation}
 A path is called rectifiable if it has finite length. Given a path $\gamma:I\to M$ and an interval $J\subset I$, we denote $\gamma_J:J\to M$ the restriction of $\gamma$ to $J$. Two paths $\gamma_i:I_i\to M$, $i=1,2$, are said to be equivalent if there exist continuous, non-decreasing and surjective functions $\varphi_i:J\to I_i$ such that 
 $\gamma_1\circ\varphi_1=\gamma_2\circ\varphi_2$. In this case, $\gamma_1$ is said to be a reparametrisation of $\gamma_2$ and one checks that $L(\gamma_1)=L(\gamma_2)$. A path $\gamma:[a,b]\to M$ is said to have constant speed if, for all $a\le s\le t\le b$,
 \[L(\gamma_{[s,t]})=\frac{t-s}{b-a}L(\gamma).\]
 If $-\infty<a< b<+\infty$ and $\gamma:[a,b]\to M$ is rectifiable, then for any $\tau>0$, $\gamma$ admits a constant speed reparametrization $[0,\tau]\to M$. A path $\gamma:[0,\tau]\to M$ is said to issue from $x$ if $\gamma(0)=x$ and is said to connect $x$ to $y$ if in addition $\gamma(\tau)=y$. It follows from the definition of length, and the triangular inequality, that $d(x,y)\le L(\gamma)$ for any path $\gamma$ connecting $x$ to $y$. In particular,
 \begin{equation}
 \label{dintrinsic0}
 d(x,y)\le \bar d(x,y):=\inf_{\gamma} L(\gamma),
\end{equation}
 where the infimum is taken over all such paths $\gamma$. The function $\bar d$ defines a $[0,+\infty]$-valued metric on $M$ called the length metric. 
 
 \begin{defi}
 The space $M$ is called a length space if $d=\bar d$. A length space $(M,d)$ is said to be a geodesic space if, in addition, the infimum in \eqref{dintrinsic0} is always attained. In a geodesic space, a path connecting $x$ to $y$ whose length is equal to $d(x,y)$ is called a shortest path connecting $x$ to $y$. A shortest path $\gamma:[0,\tau]\to M$, with constant speed, is called a geodesic. 
  \end{defi}
  
 One checks that a path $\gamma:[0,\tau]\to M$ is a geodesic iff, for all $0\le s\le t\le \tau$,
\[
d(\gamma(s),\gamma(t))=\frac{t-s}{\tau}d(\gamma(0),\gamma(\tau)).
\]
For a geodesic $\gamma:[0,\tau]\to M$, the metric speed
\[|\gamma'|:=\frac{d(\gamma(0),\gamma(t))}{t},\]
is constant by construction, for $t\in(0,\tau]$. 

For $\kappa\in \R$, a remarkable geodesic space is the $\kappa$-plane  $(M^2_{\kappa},d_{\kappa})$ defined as the unique (up to isometry) $2$-dimensional complete and simply connected Riemannian manifold with constant sectional curvature $\kappa$, equipped with its Riemannian distance $d_{\kappa}$. The diameter $D_{\kappa}$ of $M^{2}_{\kappa}$ is 
\[D_{\kappa}:=\left\{\begin{array}{cc}
     +\infty&\mbox{if}\quad\kappa\le0,\\
     \pi/\sqrt{\kappa}&\mbox{if}\quad\kappa>0.
\end{array}\right.\]
For $\kappa\in \R$, there is a unique geodesic $[0,1]\to M$ connecting $x$ to $y$ in $M^2_{\kappa}$ provided $d_{\kappa}(x,y)<D_{\kappa}$. 

Given a metric space $(M,d)$, we call triangle in $M$ any set of three points $\{p,x,y\}\subset M$. We call it non-degenerate if all three points are distinct. For $\kappa\in\R$, a comparison triangle for $\{p,x,y\}\subset M$ in $M^2_{\kappa}$ is an isometric copy $\{\bar p,\bar x,\bar y\}\subset M^2_{\kappa}$ of $\{p,x,y\}$ in $M^2_{\kappa}$ (i.e., pairwise distances are preserved). Such a comparison triangle always exists and is unique (up to an isometry) provided the perimeter \[\mathrm{peri}\{p,x,y\}:=d(p,x)+d(p,y)+d(x,y)< 2D_{\kappa}.\]
If $\{p,x,y\}$ is non-degenerate and $\mathrm{peri}\{p,x,y\}<2D_{\kappa}$, the triangular inequality implies that \[d(p,x),d(p,y),d(x,y)<D_{\kappa}.\]

Given $\kappa\in \R$, $p,x,y\in M$ with $p\notin\{x,y\}$ and $\textrm{peri}\{p,x,y\}<2D_{\kappa}$, we define the comparison angle $\sphericalangle^{\kappa}_p(x,y)\in[0,\pi]$ at $p$ by 
 \[\cos\sphericalangle^{\kappa}_p(x,y):=\left\{
 \begin{array}{ll}
 \dfrac{d^2(p,x)+d^2(p,y)-d^2(x,y)}{2d(p,x)d(p,y)}&\mbox{ if }\kappa=0,\vspace{0.2cm}\\
 \dfrac{c_{\kappa}(d(x,y))-c_{\kappa}(d(p,x))\cdot c_{\kappa}(d(p,y))}{\kappa\cdot s_{\kappa}(d(p,x))s_{\kappa}(d(p,y))}&\mbox{ if }\kappa\ne 0,
 \end{array}
 \right.\]
  where, for $r\ge 0$, we have $c_{\kappa}(r)=s'_{\kappa}(r)$ and
 \begin{equation}
 \label{def:csk}
 s_{\kappa}(r):=\left\{
 \begin{array}{ll}
 \sin(r\sqrt\kappa)/\sqrt\kappa &\mbox{ if }\kappa>0,\\
 \sinh(r\sqrt{-\kappa})/\sqrt{-\kappa}&\mbox{ if }\kappa< 0.
 \end{array}
 \right.
 \end{equation}
When $\mathrm{peri}\{p,x,y\}\ge 2D_{\kappa}$, we declare the angle $\sphericalangle^{\kappa}_p(x,y)$ undefined. Note that the comparison angle $\sphericalangle^{\kappa}_p(x,y)$ corresponds to the Riemannian angle at $\bar p$ between the two unique geodesics connecting $\bar p$ to $\bar x$ and $\bar y$ respectively in $M^2_{\kappa}$ where $\{\bar p,\bar x,\bar y\}\subset M^2_{\kappa}$ denotes a comparison triangle for $\{p,x,y\}$. 
\begin{defi}
\label{def:4pc}
Given $\kappa\in \R$, a metric space $(M,d)$ is said to have curvature lower bounded by $\kappa$, which we denote by ${\curv}(M)\ge \kappa$, if for all $p,x,y,z\in M$, such that $p\notin\{x,y,z\}$, we have
\begin{equation}
\label{eq:4pc}
\sphericalangle^{\kappa}_{p}(x,y)+\sphericalangle^{\kappa}_{p}(x,z)+\sphericalangle^{\kappa}_{p}(y,z)\le 2\pi,
\end{equation}
when all three angles are defined.
\end{defi}
Definition \ref{def:4pc} is of global nature as it requires comparison \eqref{eq:4pc} to hold for all quadruples $p,x,y,z\in M$ for which angles at $p$ are defined. A globalization result due to~\cite{bgp12} states that, when $M$ is a complete length space, then it has curvature lower bounded by $\kappa$ in the sense of Definition \ref{def:4pc} iff, for all $p\in M$, comparison \eqref{eq:4pc} holds for all $\{x,y,z\}$ in a neighborhood of $p$. In the case of geodesic spaces, we can give the following equivalent characterization of lower bounded curvature.
\begin{thm}
\label{thm:curvequiv}
Let $(M,d)$ be a geodesic space and $\kappa\in \R$. Then the following statements are equivalent.
\begin{enumerate}[label=\rm{(\arabic*)}]
\item $\curv(M)\ge \kappa$ in the sense of Definition \ref{def:4pc}.
\item\label{it:pointsonsides} For all $p,x,y\in M$ with $p\notin\{x,y\} $ and $\mathrm{peri}\{p,x,y\}< 2D_{\kappa}$, and for any geodesics $\gamma_{x},\gamma_y:[0,1]\to M$ connecting $p$ to $x$ and $p$ to $y$ respectively, we have
\[\forall s,t\in[0,1],\quad d(\gamma_x(s),\gamma_y(t))\ge d_{\kappa}( \bar\gamma_x(s),\bar\gamma_y(t)),\]
where, given a comparison triangle $\{\bar p,\bar x,\bar y\}$ of $\{p,x,y\}$ in $M^2_{\kappa}$, $\bar\gamma_x,\bar\gamma_y:[0,1]\to M^2_{\kappa}$ are geodesics (which are unique if the triangle is non-degenerate) connecting $\bar p$ to $\bar x$ and $\bar p$ to $\bar y$ respectively. 
\end{enumerate}
\end{thm}

We end the paragraph with a few standard examples of Polish geodesic spaces with lower bounded curvature in the sense of Definition \ref{def:4pc}.
\begin{itemize}
    \item A complete and connected Riemannian manifold, with its Riemannian distance, is a Polish geodesic space with curvature lower bounded by $\kappa\in \R$, iff its sectional curvatures are all lower bounded by $\kappa$.
    \item The frontier $\partial K$ of a convex and compact subset $K\subset\R^d$ (with non-empty interior) equipped with its length metric (inherited from the induced euclidean distance) is a Polish geodesic space with curvature lower bounded by $0$.
    \item Given a Polish geodesic space $(\Omega,d)$ with curvature lower bounded by $0$, the $2$-Wasserstein space $\mathcal W_2(\Omega):=(\mathcal P_2(\Omega),W_2)$ over $\Omega$ is a Polish geodesic space with curvature lower bounded by $0$.
\end{itemize}

\subsection{Spaces of directions and tangent cones} 
Let $(M,d)$ be a geodesic space with lower bounded curvature in the sense of Definition \ref{def:4pc}. 

Given $p\in M$, we denote $\Gamma_p$ be the set of all non-trivial geodesics $\gamma:[0,\tau]\to M$ issuing from $p$. For $\gamma,\sigma\in\Gamma_p$, the angle between $\gamma$ and $\sigma$ is defined by
\[\sphericalangle_p(\gamma,\sigma):=\lim_{s,t\to 0}\sphericalangle^0_p(\gamma(s),\sigma(t)).\]
The angle $\sphericalangle_p:\Gamma^2_p\to [0,\pi]$ is well defined for geodesic spaces with lower bounded curvature, as a consequence of Theorem \ref{thm:curvequiv}, point \ref{it:pointsonsides}, and satisfies, for all $\gamma, \omega, \sigma\in \Gamma_p$,
\[\sphericalangle_p(\gamma,\sigma)\le \sphericalangle_p(\gamma,\omega)+\sphericalangle_p(\omega,\sigma).\]
The angle $\sphericalangle_p$ is therefore a pseudo-metric on $\Gamma_p$ and induces a metric on the quotient space $\Sigma'_{p}:=\Gamma_{p}/\sim$ where $\gamma\sim\sigma$ iff $\sphericalangle_p(\gamma,\sigma)=0$. We denote
\[\vec{\gamma}\in \Sigma'_p\] 
the equivalence class of $\gamma\in\Gamma_p$ for $\sim$. The completion $\Sigma_p$ of $\Sigma'_{p}$ is called the space of directions at $p$. Below we use the same symbol $\sphericalangle_p$ to denote the pseudo-metric on $\Gamma_p$, the metric on $\Sigma'_p$ or the metric on $\Sigma_p$.

Given a metric space $(\Omega,d)$, with diameter at most $\pi$, consider the equivalence relation $\approx$ on $\Omega\times\R_+$ defined by $(p,s)\approx(q,t)$ iff ($s=t=0$ or $(p,s)=(q,t)$). In other words, if $[p,s]$ denotes the class of $(p,s)$ for this relation, then $[p,s]=\{(p,s)\}$ if $s>0$ and $[p,0]=\Omega\times\{0\}$. The Euclidean cone over $\Omega$, denoted ${\rm cone}(\Omega)$, is the quotient set $(\Omega\times\R_+)/\approx$ equipped with the metric $d_c$ defined by 
\[d^2_c([p,s],[q,t]):=s^2-2st\cos d(p,q)+t^2.\]
We call $[p,0]$ the tip of the cone. 

The tangent cone $T_pM$ of $M$ at $p$ is defined as the Euclidean cone over the space of directions of $M$ at $p$, i.e.,
 \[T_pM:={\rm cone}(\Sigma_p).\]
 We denote $\|.-.\|_p$ the metric on $T_pM$ and $0_p$ the tip of $T_pM$. For $u=[\xi,s]\in T_p M$ and $\lambda\in\R_+$, we define $\lambda u:=[\xi,\lambda s]$. For $u=[\xi,s],v=[\zeta,t]\in T_p M$, we set $\|u\|_p:=\|u-0_p\|_p$ and $\langle u,v\rangle_p:= st\cos \sphericalangle_p(\xi,\zeta)$ so that 
\begin{equation}
\nonumber  
\|u-v\|^2_p=\|u\|^2_p-2\langle u,v\rangle_p+\|v\|^2_p.
\end{equation}

A useful alternative representation of $T_pM$ is obtained as follows. For any two geodesics $\gamma,\sigma\in \Gamma_p$, denote
\[|\gamma-\sigma|^2_p:=\lim_{t\to 0}\frac{d^2(\gamma(t),\sigma(t))}{t^2}.\]
Angles between elements of $\Gamma_p$ being well defined, the limit always exists and $|.-.|_p$ defines a pseudo metric on $\Gamma_p$. Denoting $\propto$ the equivalence relation on $\Gamma_p$ defined by $\gamma\propto\sigma$ iff $|\gamma-\sigma|_p=0$, we define $T'_pM$ as the quotient set $\Gamma_p/\propto$ equipped with the induced metric $|.-.|_p$. For $\gamma\in\Gamma_p$, we denote \[\dot\gamma\in T'_pM\] 
its class for relation $\propto$. 
\begin{lem}
\label{lem:ctpm}
The map  
\begin{equation}
\label{eq:themap}
\dot\gamma\in T'_pM\mapsto [\vec{\gamma},|\gamma'|]\in {\rm Cone}(\Sigma'_p), 
\end{equation}
is a well defined isometry and the completion of $T'_pM$ is isometric to $T_pM$.
\end{lem}
We report the proof of Lemma \ref{lem:ctpm} in the appendix. 

From now on, we'll therefore consider $T'_pM$ as a dense subset of $T_pM$ and identify $\dot\gamma\in T'_pM$ to the element $[\vec{\gamma},|\gamma'|]\in T_pM$. 

\subsection{Logarithmic map} 
Suppose that $(M,d)$ is a geodesic space with lower bounded curvature. For $p\in M$, we call logarithmic map at $p$ any map $\log_p:M\to T_pM$ such that, for all $x\in M$, 
\[\log_p(x)=\dot\gamma_x,\]
for some geodesic $\gamma_x:[0,1]\to M$ connecting $p$ to $x$. The next result shows that the choice of a sufficiently well behaved log map is possible provided $M$ is a Polish geodesic space. It was first cited as a remark in~\cite{GouParRigStr19} and proved in~\cite{Gou20}. 
\begin{lem}
\label{lem:logismeas}
Let $(M,d)$ be a Polish geodesic space with lower bounded curvature and equipped with its Borel $\sigma$-algebra. Then, for all $p\in M$, there exists a logarithmic map $\log_p:M\to T_pM$ which is measurable when $T_pM$ is equipped with the $\sigma$-algebra generated by open balls. 
\end{lem}
We report the proof of Lemma \ref{lem:logismeas} in the appendix. 
\begin{rem}
\label{rem:sketchofproof}
The proof of Lemma \ref{lem:logismeas} shows that, under these assumptions, we can first select a collection $(\gamma_x)_{x\in M}$ of geodesics $\gamma_x:[0,1]\to M$ connecting $p$ to $x$, such that the map $x\in M\mapsto \gamma_x\in G_p$ is Borel measurable when we equip $G_p\subset \Gamma_p$, the set of all geodesics issuing from $p$ and defined on $[0,1]$, with the uniform metric. Then, we show that the map $\gamma\in G_p\mapsto \dot\gamma\in T_pM$ is measurable when $T_pM$ is equipped with the $\sigma$-algebra generated by open balls. We finally define $\log_p$ as the composition of these two maps. 
\end{rem}

It follows from Lemma \ref{lem:logismeas} that we can choose a Borel-measurable logarithmic map at any point $p$ whenever the tangent cone $T_pM$ is separable since in this case the Borel $\sigma$-algebra on $T_pM$ coincides with the $\sigma$-algebra generated by open balls. This occurs for instance in the case where $M$ is a proper metric space as noted in~\cite{ohta2012barycenters}. It is also known to be the case for specific examples of non-proper spaces. For instance, if $\Omega$ denotes a Polish geodesic space with curvature lower bounded by $0$, the $2$-Wasserstein space $\mathcal W_2(\Omega)$ has a separable tangent cone at any point. This fact follows from~\citet[Definition 12.4.3]{AmbGigSav08} which characterises $T_p\mathcal W_2(\Omega)$ as a closed subset of a separable metric space.

 However, it should be noted that the measurability of $\log_p$ with respect to the $\sigma$-algebra generated by open balls on $T_pM$ is enough for the results we present next. Indeed, statements presented below require only the Borel-measurability of maps of the form
 \[x\mapsto\langle\log_p(x),u\rangle_p,\]
 for some fixed $p\in M$ and $u\in T_pM$, which follows from this weaker measurability of $\log_p$.
 
Note finally that the choice of a measurable log is in principle not unique. However, all results we'll mention can be shown to hold independently of its choice. 

Next is the first key result for the proof of Theorem \ref{thm:main}.
\begin{lem}
\label{lem:tpstarisflat}
Let $(M,d)$ be a Polish geodesic space with lower bounded curvature in the sense of Definition \ref{def:4pc}. Let $\mu\in \mathcal P_2(M)$ and suppose it admits a barycenter $x^*\in M$. Then for all $u\in T_{x^*}M$,
     \[\int_M\langle\log_{x^*}(x),u\rangle_{x^*}\,{\rm d}\mu(x)=0.\]
\end{lem}
Lemma \ref{lem:tpstarisflat} follows by combining~\citet[Theorem 7]{GouParRigStr19} and~\citet[Corollary 2]{Gou20}.

\subsection{Differential and gradient of semi-concave functions}
 Suppose $(M,d)$ is a geodesic space with lower bounded curvature in the sense of Definition \ref{def:4pc}. 

A function $f:M\to \R$ is called (geodesically) $\alpha$-concave, for some $\alpha\in\R$, if for any geodesic $\gamma:[0,1]\to M$, the map
\[t\in[0,1]\mapsto f(\gamma(t))-\frac{\alpha}{2}d^2(\gamma(0),\gamma(t)),\]
is concave. Hence, $f$ is $\alpha$-concave iff $(-f)$ is $(-\alpha)$-convex. 

A function $f:M\to \R$ is called locally Lipschitz at $p$ if there exists a constant $\lambda>0$ such that inequality \[|f(x)-f(y)|\le \lambda d(x,y)\] 
holds for all $x,y$ in some neighborhood of $p$. We denote ${\rm Lip}_p(f)$ the smallest such constant. 

Let $f:M\to \R$ be an $\alpha$-concave function, locally Lipschitz at $p\in M$. Let $d_p f:T'_pM\to \R$ be defined by
\begin{equation}
\label{defi:dpf}
d_pf(\dot\gamma):=\lim_{t\to 0}\frac{f(\gamma(t))-f(p)}{t}.
\end{equation}
\begin{lem}
\label{lem:propdiff}
Let $\alpha\in \R$ and $f:M\to \R$ be an $\alpha$-concave function, locally Lipschitz at $p\in M$. Then the limit in \eqref{defi:dpf} is well defined. For any geodesic $\gamma:[0,\tau]\to M$ issuing from $p$, this limit can be written
\begin{equation}
\label{eq:diffasinf}
\lim_{t\to 0}\frac{f(\gamma(t))-f(p)}{t}=\sup_{t\in(0,\tau]}\left\{\frac{f(\gamma(t))-f(p)}{t}-\frac{\alpha t}{2}|\gamma'|^2\right\},
\end{equation}
and does not depend on the representative $\gamma$ of $\dot\gamma$. Furthermore, the map $d_pf:T'_pM\to \R$ admits a unique ${\rm Lip}_p(f)$-Lipschitz extension to $T_pM$, which we also denote $d_pf$, and call the differential of $f$ at $p$. Finally, $d_pf:T_pM\to \R$ is positively homogeneous, i.e., satisfies, for all $\lambda\ge 0$ and all $v\in T_pM$, 
\[d_pf(\lambda v)=\lambda d_pf(v).\]
\end{lem}

We prove Lemma \ref{lem:propdiff} in the appendix.

Given an $\alpha$-concave function $f:M\to \R$, we call gradient of $f$ at $p$ any element $g\in T_pM$ such that, for all $v\in T_pM$,
\[d_pf(v)\le \langle g,v\rangle_p\quad\mbox{and}\quad d_pf(g)=\|g\|^2_p.\]

The existence of gradients for $\alpha$-concave functions, in spaces with lower bounded curvature, is the second essential result for the proof of Theorem \ref{thm:main}.
\begin{lem}[\citealp{AleKapPet19}, Theorem 11.4.2]
\label{lem:grad}
Let $(M,d)$ be a geodesic space with lower bounded curvature. Let $\alpha\in \R$ and $f:M\to \R$ be an $\alpha$-concave function, locally Lipschitz at $p$. Then there exists a unique gradient of $f$ at $p$, denoted $\nabla f(p)$.
\end{lem}
We include the proof of Lemma \ref{lem:grad} in the appendix.

\section{Proof of Theorem \ref{thm:main}}
\label{sec:theproof}

For all $x\in M$, let $\gamma_x:[0,1]\to M$ be a geodesic such that $\gamma_x(0)=x^*$ and $\gamma_x(1)=x$. Suppose \[(\gamma_x)_{x\in M}\] 
is chosen as indicated in Remark \ref{rem:sketchofproof}, for $p=x^*$, and denote $\log_{x^*}:M\to T_{x^*}M$ the corresponding logarithmic map at $x^*$. By $\alpha$-convexity of $f$, we deduce that, for all $t\in(0,1]$ and all $x\in M$,
\[f(\gamma_x(t))\le (1-t)f(x^*)+tf(x)-\frac{\alpha}{2}t(1-t)d^2(x^*,x).\]
Rearranging terms, we get
\begin{align*}
f(x^*)&\le f(x)-\frac{f(\gamma_x(t))-f(x^*)}{t}-\frac{\alpha}{2}(1-t)d^2(x^*,x)\\
&=f(x)+\frac{(-f)(\gamma_x(t))-(-f)(x^*)}{t}-\frac{\alpha}{2}(1-t)d^2(x^*,x).
\end{align*}
Notice that $(-f):M\to \R$ is $(-\alpha)$-concave. Hence, taking the limit $t\to 0$, it follows by definition of the differential, and the fact that \[\dot\gamma_x=\log_{x^*}(x),\] 
that
\[f(x^*)\le f(x)+d_{x^*}(-f)(\log_{x^*}(x))-\frac{\alpha}{2}d^2(x^*,x).\]
Now by Lemma \ref{lem:grad}, and the definition of $\nabla(-f)(x^*)$, we deduce that
\[f(x^*)\le f(x)+\langle\log_{x^*}(x),\nabla(-f)(x^*)\rangle_{x^*}-\frac{\alpha}{2}d^2(x^*,x).\]
The result follows by integrating both sides with respect to $\mu$ and using Lemma \ref{lem:tpstarisflat}.

\appendix
\section{Appendix}
\label{sec:appendix}

\subsection*{Proof of Lemma \ref{lem:ctpm}}
Direct computations reveal that, for all $\gamma, \sigma\in \Gamma_p$, 
\[\frac{d^2(\gamma(t),\sigma(t))}{t^2}=|\gamma'|^2+|\sigma'|^2-2|\gamma'||\sigma'|\cos\sphericalangle^0_p(\gamma(t), \sigma(t)),\]
provided $t>0$ is small enough that $\gamma(t)$ and $\sigma(t)$ are defined.
Taking the limit $t\to 0$, it follows by definition of angles that
\begin{align}
|\gamma-\sigma|^2_p&=|\gamma'|^2+|\sigma'|^2-2|\gamma'||\sigma'|\cos\sphericalangle_p(\gamma, \sigma)
\label{lem:ctpm:e1}\\
&=(|\gamma'|-|\sigma'|)^2+2|\gamma'||\sigma'|(1-\cos\sphericalangle_p(\gamma, \sigma)).
\label{lem:ctpm:e2}
\end{align}
Identity \eqref{lem:ctpm:e2} shows that $|\gamma-\sigma|_p$ is indeed well defined and that $\dot\gamma=\dot\sigma$ iff $|\gamma'|=|\sigma'|$ and $\vec{\gamma}=\vec{\sigma}$. Hence, the map \eqref{eq:themap} is defined without ambiguity, as it doesn't depend on particular representatives, and is injective. In addition, expression \eqref{lem:ctpm:e1} translates precisely as
\[|\dot\gamma-\dot\sigma|_p=d_c([\vec{\gamma},|\gamma'|],[\vec{\sigma},|\sigma'|]),\] 
where $d_c$ denotes the metric in ${\rm cone}(\Sigma'_p)$. This proves that the map \eqref{eq:themap} is distance preserving. Finally, this map is surjective since, for any $s>0$ and any $\gamma:[0,\tau]\to M$ in $\Gamma_p$, $[\vec{\gamma},s]\in {\rm cone}(\Sigma'_p)$ is the image of $\dot\gamma_{\alpha}$ where $\gamma_{\alpha}:[0,\tau/\alpha]\to M$ is defined by $\gamma_{\alpha}(t)=\gamma(\alpha t)$ with $\alpha:=|\gamma'|/s$. To show that the completion of $T'_pM$ is isometric to $T_pM$ it remains to observe that, more generally, the cone over the completion of a metric space is isometric to the completion of the cone over that space. 

To prove this statement, consider a metric space $(\Omega,d)$ and denote $(\bar{\Omega},\bar{d})$ its completion. We denote $\bar x_n$ the equivalence class of Cauchy sequence $x_n$ for the equivalence relation  $\lim_n d(x_n,y_n)=0$ and understand $\bar d$ as 
\[\bar{d}(\bar{x}_n,\bar{y}_n)=\lim_n d(x_n,y_n).\]
Writing 
\[d^2_c([p,s],[q,t])=(s-t)^2+2st(1-\cos d(p,q)),\]
we see that a sequence $[p_n,s_n]$ is a Cauchy sequence in ${\rm cone}(\Omega)$ iff ($s_n\to 0$) or ($s_n\to s_{\infty}>0$ and $p_n$ is a Cauchy sequence in $\Omega$).
As a result, the map $\phi:{\rm cone}(\bar\Omega)\to \bar{{\rm cone}}(\Omega)$ defined by
\[\phi([\bar p_n,s]):=\left\{\begin{array}{cc}
\overline{[p_n,s]}&\mbox{ if }s>0,\\
     0&\mbox{ if }s=0,
\end{array}\right.\]
is well defined and distance preserving when both $\bar\Omega$ and $\bar{{\rm cone}}(\Omega)$ are equipped with the completion metric. One checks finally that it is invertible with inverse given by $\phi^{-1}( \overline{[p_{n},s_{n}]})=0$ if $s_n\to 0$ and $\phi^{-1}( \overline{[p_{n},s_{n}]})=[\bar{p}_n,s_{\infty}]$ if $s_{n}\to s_{\infty}>0$.

\subsection*{Proof of Lemma \ref{lem:logismeas}}

We simply report and detail the proofs of Lemmas 3.3 and 4.2 in~\citet{ohta2012barycenters} emphasizing that they actually do not require the metric space $M$ to be proper, an assumption imposed in~\citet{ohta2012barycenters} for other reasons.

Suppose that $(M,d)$ is a Polish geodesic space and has lower bounded curvature. Introduce the set $G_p\subset \Gamma_p$ of all geodesics $\gamma:[0,1]\to M$ issuing from $p$. Equipped with the supremum metric 
\[d_{\infty}(\gamma,\sigma):=\sup_{t\in[0,1]}d(\gamma(t),\sigma(t)),\]
$G_p$ is a Polish metric space since it is closed in the Polish metric space $(G,d_{\infty})$ of all geodesics $\gamma:[0,1]\to M$ equipped with the supremum metric~\cite[Proposition 1.2.1]{GigPas20}.  

For $t\in[0,1]$, denote $e_t:G_p\to M$ the evaluation map defined by $e_t(\gamma):=\gamma(t)$. This evaluation map is (Lipschitz) continuous. Hence, for all $x\in M$, the set $e^{-1}_1(x)\subset G_p$ is closed and non-empty. Furthermore, for any open set $U\subset G_p$, the set
\[\{x\in M: e^{-1}_1(x)\cap U\ne\emptyset\}=e_1(U),\]
is a Borel set. 

Indeed, fix a non-empty open set $U\subset G_p$. For $\delta>0$, denote 
\[A_\delta:=G_p\setminus\{\gamma\in G_p: d_{\infty}(\gamma,G_p\setminus U)<\delta\}.\]
The set $A_{\delta}$ is closed in $G_p$, satisfies $A_\delta\subset A_{\delta'}$ iff $\delta'\le \delta$ and is such that $\cup_{\delta>0}A_{\delta}=U$. Given $\e,\delta>0$, consider the set $M_{\e,\delta}\subset M$ of all points $x$ for which there exists a rectifiable path $\sigma:[0,1]\to M$ satisfying $\sigma(0)=p$, $\sigma(1)=x$ and
\[d_{\infty}(\sigma,A_{\delta})<\e.\]
The set $M_{\e,\delta}$ is open in $M$ and satisfies $M_{\e,\delta}\subset M_{\e',\delta}$ iff $\e\le\e'$. Hence, the set $e_1(A_{\delta})=\cap_{\e>0}M_{\e,\delta}$ is a Borel set in $M$ since the intersection can be restricted to any countable sequence $\e_n\downarrow 0$ by monotonicity. As a result, the set $e_1(U)=\cup_{\delta>0}e_1(A_{\delta})$ is a Borel set in $M$ as well since, again, the reunion can be restricted to any countable sequence $\delta_n\downarrow 0$ by monotonicity.

Now, by application of the Kuratowski and Ryll-Nardzewski measurable selection theorem~\citep[Theorem 6.9.3]{Bog07}, there exists a measurable function $g:M\to G_p$ such that, for all $x\in M$, $g(x)\in e^{-1}_1(x)$. We denote \[\gamma_x:=g(x).\]

It remains to show that the map $\theta_p:\gamma\in G_p\mapsto\dot\gamma\in T_pM$ is measurable when $T_pM$ is equipped with the $\sigma$-algebra generated by open balls. To prove this fact, fix $v\in T_pM$ and $r>0$. By density of $T'_pM$ in $T_pM$, there exists $\gamma_n\in G_p$ such that $\dot\gamma_n\to v$. Therefore,
\begin{align*}
    \theta^{-1}_p(B(v,r))&=\{\gamma\in G_p:\|\dot\gamma-v\|_p<r\}\\
    &= \bigcup_{n\ge 1}\bigcap_{m\ge n}\{\gamma\in G_p:\|\dot\gamma-\dot\gamma_m\|_p<r\}\\
    &= \bigcup_{n\ge 1}\bigcap_{m\ge n}\{\gamma\in G_p:\lim_{t\to 0}t^{-1}d(\gamma(t),\gamma_m(t))<r\}\\
    &=\bigcup_{n\ge 1}\bigcap_{m\ge n}\bigcup_{\ell\ge 1}\bigcap_{k\ge \ell}\{\gamma\in G_p:d(\gamma(1/k),\gamma_m(1/k))<r/k\}.\\
\end{align*}
But since, for fixed $k,m\ge 1$, the set 
$\{\gamma\in G_p:d(\gamma(1/k),\gamma_m(1/k))<r/k\}$
is open in $(G_p,d_{\infty})$, we see that $\theta^{-1}_p(B(v,r))$ is a Borel subset of $(G_p,d_{\infty})$ which completes the proof.

\subsection*{Proof of Lemma \ref{lem:propdiff}}

Let $\gamma:[0,\tau]\to M$ be any geodesic issuing from $p$ and let $\sigma:[0,1]\to M$ be the reparametrization of $\gamma$ defined by $\sigma(t)=\gamma(\tau t)$. Then, for all $t\in (0,1]$,
\begin{align}
    \frac{f(\gamma(\tau t))-f(p)}{\tau t}&=\frac{1}{\tau}\left\{\frac{f(\sigma(t))-f(p)}{t}-\frac{\alpha t}{2}|\sigma'|^2\right\}+ \frac{\alpha t}{2\tau}|\sigma'|^2
    \nonumber\\
    &=\frac{1}{\tau}\left\{\frac{g(t)-g(0)}{t}\right\}+ \frac{\alpha t}{2\tau}|\sigma'|^2,
    \label{lem:propdiff:e1}
\end{align}
where $g:[0,1]\to \R$ is defined by
\[g(t)=f(\sigma(t))-\frac{\alpha t^2}{2}|\sigma'|^2=f(\sigma(t))-\frac{\alpha}{2}d^2(p,\sigma(t)).\]
The $\alpha$-concavity of $f$ implies that $g$
is concave on $[0,1]$. Hence, the right derivative of $g$ at $0$ is well defined and satisfies 
\begin{align}
\lim_{t\downarrow 0}\frac{g(t)-g(0)}{t}&=\sup_{t\in(0,1]}\frac{g(t)-g(0)}{t}
\nonumber\\
&= \sup_{t\in(0,1]}\left\{\frac{f(\sigma(t))-f(p)}{t}-\frac{\alpha t }{2}|\sigma'|^2\right\}.
\label{lem:propdiff:e2}
\end{align}
Hence, combining \eqref{lem:propdiff:e1}, \eqref{lem:propdiff:e2} and noticing that $|\sigma'|=\tau|\gamma'|$, we obtain
\begin{align*}
    \lim_{t\downarrow 0}\frac{f(\gamma(t))-f(p)}{t}&=\lim_{t\downarrow 0}\frac{f(\gamma(\tau t))-f(p)}{\tau t}\\
    &= \frac{1}{\tau}\sup_{t\in(0,1]}\left\{\frac{f(\sigma(t))-f(p)}{t}-\frac{\alpha t }{2}|\sigma'|^2\right\}\\
    &= \frac{1}{\tau}\sup_{t\in(0,1]}\left\{\frac{f(\gamma(\tau t))-f(p)}{t}-\frac{\alpha \tau^2t}{2}|\gamma'|^2\right\}\\
    &= \sup_{t\in(0,\tau]}\left\{\frac{f(\gamma(t))-f(p)}{t}-\frac{\alpha t}{2}|\gamma'|^2\right\},
\end{align*}
which proves \eqref{eq:diffasinf}.
Now, for any $\dot\gamma,\dot\sigma\in T'_pM$, the local Lipschitz property of $f$ implies that
\begin{align*}
|d_pf(\dot\gamma)-d_pf(\dot\sigma)|&=\lim_{t\to 0}\frac{|f(\gamma(t))-f(\sigma(t))|}{t}\\ 
&\le {\rm Lip}_p(f)\lim_{t\to 0}\frac{d(\gamma(t),\sigma(t))}{t}\\
&={\rm Lip}_p(f)\|\dot\gamma-\dot\sigma\|_p.
\end{align*}
On the one hand, this inequality shows that the limit on the right hand side of \eqref{defi:dpf} is independent of the chosen representative $\gamma$ of $\dot\gamma$. On the other hand, by density of $T'_pM$ in $T_pM$, it implies that $d_pf$ admits a unique ${\rm Lip}_p(f)$-Lipschitz extension to $T_pM$.  

Finally, to prove $d_pf$ is positively homogeneous on $T_pM$ it is enough to show it on $T'_pM$ and conclude by continuity. But, for any geodesic $\gamma:[0,\tau]\to M$ and any $\lambda>0$, the geodesic $\sigma:[0,\lambda\tau]\to M$ defined by $\sigma(t)=\gamma(t/\lambda)$ satisfies $\dot\sigma=\lambda\dot\gamma$ and we see that 
\[d_pf(\dot\sigma)=\lim_{t\to 0}\frac{f(\sigma(t))-f(p)}{t}=\lambda\lim_{t\to 0}\frac{f(\gamma(t/\lambda))-f(p)}{t/\lambda}=\lambda d_pf(\dot\gamma),\]
which completes the proof.

\subsection*{Proof of Lemma \ref{lem:grad}}
The proof follows the lines devised in~\citet{AleKapPet19} with minor modifications. In particular, we use the following result, due to~\citet{LanSch97}.  

\begin{lem}[\citealp{LanSch97}, Lemma A.4]
\label{lem:ls}
Let $(M,d)$ be geodesic with lower bounded curvature, in the sense of Definition \ref{def:4pc}. Let $p\in M$ and $\gamma,\sigma:[0,1]\to M$ two geodesics issuing from $p$. For all $t\in(0,1]$, let $\delta_t:[0,1]\to M$ be a geodesic connecting $\gamma(t)$ to $\sigma(t)$. Introducing the midpoint $m_t=\delta_t(1/2)$ of $\gamma(t)$ and $\sigma(t)$, we have
\[4\lim_{t\downarrow 0}\frac{d^2(p,m_t)}{t^2}=\|\dot\gamma\|^2_p+2\langle\dot\gamma,\dot\sigma\rangle_p+\|\dot\sigma\|^2_p.\]
\end{lem}

We can now establish the following result.

\begin{lem}[\citealp{AleKapPet19}, Lemma 11.2.3]
\label{lem:dpfisconv}
Let $(M,d)$ be geodesic with lower bounded curvature, in the sense of Definition \ref{def:4pc}. Suppose $f:M\to\R$ is locally Lipschitz at $p$ and $\alpha$-concave. Then for all $u,v\in T_pM$, 
\[\sup_{w\in T_pM:\|w\|_p=1}d_pf(w)\ge \frac{d_pf(u)+d_pf(v)}{\sqrt{\|u\|^2_p+2\langle u,v\rangle_p+\|v\|^2_p}}.\]
\end{lem}
\begin{proof}[Proof of Lemma \ref{lem:dpfisconv}]
By density of $T'_pM$ in $T_pM$ and continuity of $d_pf$, it is enough to prove the result for all $u=\dot\gamma\in T'_pM$ and $v=\dot\sigma\in T'_pM$ where $\gamma,\sigma\in \Gamma_p$. For two such geodesics $\gamma:[0,1]\to M$ and $\sigma:[0,1]\to M$, and all $t\in(0,1]$, let $\delta_{t}:[0,1]\to M$ be a geodesic connecting $\gamma(t)$ to $\sigma(t)$. Then, by concavity of $f$,  
\begin{equation}
\label{lem:dpfisconv:e2} 
2f(\delta_{t}(1/2))\ge f(\gamma(t))+f(\sigma(t))-\frac{\alpha}{4}d^2(\gamma(t),\sigma(t)).
\end{equation}
Observing that 
\begin{equation}
\nonumber
f(\gamma(t))=f(p)+ td_pf(\dot\gamma)+{\rm o}(t),\quad f(\sigma(t))=f(p)+td_pf(\dot\sigma)+{\rm o}(t),
\end{equation}
and that
\begin{equation}
\nonumber
d^2(\gamma(t),\sigma(t))=t^2\|\dot\gamma-\dot\sigma\|^2_p+{\rm o}(t^2),
\end{equation}
we deduce from \eqref{lem:dpfisconv:e2} that
\[\liminf_{t\downarrow0}\left(\frac{f(\delta_{t}(1/2))-f(p)}{t}\right)\ge  \frac{d_pf(\dot\gamma)+d_pf(\dot\sigma)}{2}.\]
Now, for all $t\in(0,1]$, let $\omega_t:[0,t]\to M$ be a geodesic connecting $p$ to $\delta_t(1/2)$. Then, according to Lemma \ref{lem:propdiff}, we see that
\begin{align*}
    \frac{f(\delta_{t}(1/2))-f(p)}{t}&=\frac{f(\omega_{t}(t))-f(p)}{t}\\
    &\le \sup_{s\in(0,t]}\left\{\frac{f(\omega_{t}(s))-f(p)}{s}-\frac{\alpha s}{2}|\omega'_t|^2\right\} +\frac{\alpha t}{2}|\omega'_t|^2\\
    &= d_pf(\dot\omega_t)+\frac{\alpha t}{2}|\omega'_t|^2\\
    &=|\omega'_t|\left(d_pf(|\omega'_t|^{-1}\dot\omega_t)+\frac{\alpha t}{2}|\omega'_t|\right).
\end{align*}
Since $\||\omega'_t|^{-1}\dot\omega_t\|_p=1$, for all $t\in(0,1]$, it remains only to prove that
\[4|\omega'_t|^2= \|\dot\gamma\|^2_p+2\langle\dot\gamma,\dot\sigma\rangle_p+\|\dot\sigma\|^2_p+\e(t),\]
where $\e(t)\to 0$ as $t\to 0$. But this follows by observing that, 
\begin{align*}
    4\lim_{t\downarrow 0}|\omega'_t|^2&=4\lim_{t\downarrow 0}\frac{d^2(p,\omega_t(t))}{t^2}\\
    &=\|\dot\gamma\|^2_p+2\langle\dot\gamma,\dot\sigma\rangle_p+\|\dot\sigma\|^2_p,
\end{align*}
according to Lemma \ref{lem:ls}.
\end{proof}

We are now in position to prove existence and uniqueness of gradients for $\alpha$-concave functions.
\begin{proof}[Proof of Lemma \ref{lem:grad}]
 Let $w_n=[\xi_n,1]\in T_pM$ be such that
\[\lim_{n\to+\infty} d_pf(w_n)=d_{\rm sup}:=\sup_{w\in T_pM:\|w\|_p=1}d_pf(w).\]
Applying Lemma \ref{lem:dpfisconv}, we get for all $m,n\ge 1$,
\[d_{\rm sup}\ge \frac{d_pf(w_n)+d_pf(w_m)}{\sqrt{2+2\langle w_n,w_m\rangle_p}}=\frac{d_pf(w_n)+d_pf(w_m)}{\sqrt{2+2\cos\sphericalangle_p(\xi_n,\xi_m)}}.\]
Letting $n,m\to+\infty$, this implies that $\xi_n$ is a Cauchy sequence in the complete space $\Sigma_p$, and hence converges towards some $\xi_{\infty}\in\Sigma_p$. Denoting $w_{\infty}=[\xi_{\infty},1]$, we have $d_pf(w_{\infty})=d_{\rm sup}$ by continuity. Now, denote \[g:= d_{\rm sup}w_{\infty},\] 
and select an arbitrary $w\in T_pM$. Then applying Lemma \ref{lem:dpfisconv} to $u=w_{\infty}$ and $v=\e w$, we get
\[d_{\rm sup}\ge \frac{d_{\rm sup}+\e d_pf(w)}{\sqrt{1+2\e\langle w_{\infty},w\rangle_p+\e^2\|w\|^2_p}}=d_{\rm sup}+\e(d_pf(w)-\langle g,w\rangle_p)+{\rm o}(\e).\]
Letting $\e\downarrow 0$, we obtain  
$d_pf(w)\le \langle g,w\rangle_p$. Since it is clear by construction that $d_pf(g)=\|g\|^2_p$, this proves the existence of a gradient. To prove uniqueness, consider another $g'$ satisfying the same properties. Than, we have
\[\|g'\|^2_p=d_pf(g')\le\langle g,g'\rangle_p\quad\mbox{and}\quad \|g\|^2_p=d_pf(g)\le\langle g,g'\rangle_p.\]
As a result, 
\[0\le\|g'-g\|^2_p=\|g'\|^2_p-2\langle g,g'\rangle_p+\|g\|^2\le 0,\]
which imposes $g=g'$ and completes the proof.
\end{proof}

\textbf{Acknowledgments}. I would like to thank Thibaut Le Gouic and Philippe Rigollet for discussions on Metric Geometry that sparked my interest for the problem addressed in the paper.

\bibliography{Jensen}
\end{document}